\documentclass{amsart}

\usepackage{amssymb}
\usepackage{color}
\usepackage{epsfig}
\usepackage{calrsfs}
\usepackage[shortlabels]{enumitem}
\usepackage[hidelinks]{hyperref}

\theoremstyle{plain}
\begingroup

\newtheorem*{theoremB'}{Theorem B'}
\newtheorem*{theoremB''}{Theorem B''}

\newtheorem{theorem}{Theorem}[section]

\newtheorem{lemma}[theorem]{Lemma}
\newtheorem{proposition}[theorem]{Proposition}

\newtheorem{definition}[theorem]{Definition}

\newtheorem{example}[theorem]{Example}

\endgroup

\theoremstyle{remark}
\begingroup
\endgroup

\numberwithin{equation}{section}

\newcommand{\curr}[1]{[\![{#1}]\!]}

\title{Coefficient groups inducing nonbranched optimal transport}

\author{Mircea Petrache}
\address{Max-Planck Institute for Mathematics, Vivatsgasse 7, 53111 Bonn, Germany}
\email{decostruttivismo@gmail.com}
\author{Roger Z\"{u}st}
\address{University of Bern, Mathematical Institute, Alpeneggstrasse 22, 3012 Bern, Switzerland}
\email{roger.zuest@math.unibe.ch}



\begin{document}

\begin{abstract}
In this work we consider an optimal transport problem with coefficients in a normed Abelian group $G$, and extract a purely intrinsic condition on $G$ that guarantees that the optimal transport (or the corresponding minimum filling) is not branching. The condition turns out to be equivalent to the nonbranching of minimum fillings in geodesic metric spaces. We completely characterize finitely generated normed groups and finite-dimensional normed vector spaces of coefficients that induce nonbranching optimal transport plans. We also provide a complete classification of normed groups for which the optimal transport plans, besides being nonbranching, have acyclic support. This seems to initiate a new geometric classifications of certain normed groups. In the nonbranching case we also provide a global version of calibration, i.e.\ a generalization of Monge-Kantorovich duality.
\end{abstract}

\maketitle

\section{Introduction}

\subsection{Basic setting and motivation}
The present work can be considered as an attempt to do an \emph{ab initio} study of transportation problems, interpreted in a very broad sense. We consider $n$ points $x_1,\ldots, x_n$ in a space $X$, and associated coefficients $g_1,\ldots,g_n$ in a space $G$.

\medskip 

\noindent These points and coefficients may be interpreted as locations and quantifications of some entities. Then, informally speaking, we want to study the properties of the ``lowest cost $1$-dimensional transport'' for the quantities $g_i$ between the sources $x_i$, under minimal assumptions on $X$ and $G$. For more geometric motivations to the same problems see also the introductions of \cite{fleming}, \cite{white}, \cite{whitedef}, \cite{marchesemasaccesi}.

\medskip

\noindent It is a natural assumption to require $X$ to be a \textbf{geodesic metric space}. As we want to be able to implement a ``lossless transport'' condition for the quantity modeled by $G$ we have to be able to combine together different quantities $g_i$, and thus the space $G$ has to be a group. Because at a crossing of our transport system the order in which we sum contributions from the different branches is irrelevant, we require $G$ to be an Abelian group. Moreover to compare different coefficients we consider a norm $|\cdot|_G:G\to\mathbb R_{\ge 0}$ compatible with the group operation, i.e. we consider a \textbf{normed Abelian group} $(G,|\cdot|_G)$. Following \cite{white}, the axioms for $|\cdot|_G$ are
\begin{enumerate}
	\item $|g + h|_G \leq |g|_G + |h|_G$ for all $g,h \in G$,
	\item $|g|_G = |-g|_G$ for all $g \in G$,
	\item $|g|_G = 0$ if and only if $g = 0_G$.
\end{enumerate}
We often write $|\cdot|$ and $0$ rather than $|\cdot|_G$ and $0_G$, in case the group and norm are clear from the context.

\medskip

\noindent The optimal transport problem which we consider is the following: Given points $x_1,\ldots,x_n\in X$ and coefficients $g_1,\ldots,g_n\in G$, we split each $g_i$, $i=1,\ldots,n$, into parts $g_{ij}$, $j=1,\ldots,n$, and then interpret $g_{ij}$ to be the quantity moving from $x_i$ to $x_j$ ($g_{ij}$ is then assumed to be equal to $-g_{ji}$). Among all such decompositions we seek the one minimizing the transport cost $\sum_{i<j}|g_{ij}|d_X(x_i,x_j)$.

\begin{definition}[Optimal transport plans]
\label{optitransp}
Let $g_1,\dots,g_n$ be elements in a normed Abelian group $G$ such that $g_1+\dots+g_n = 0$ and assume that an $n$-point metric space is given by $(\{x_1,\dots,x_n\},d_X)$. Then we set
\begin{equation}
\label{optplans}
OT\biggl(\sum_{i=1}^ng_i\curr{x_i}\biggr) \mathrel{\mathop:}= \inf \sum_{1 \leq i < j \leq n}|g_{ij}|d_X(x_i,x_j) \ ,
\end{equation}
where the infimum is taken over all $g_{ij}$, $i,j = 1,\dots,n$, with $g_{ij}=-g_{ji}$, $g_{ii} = 0$ and $g_i=\sum_{j=1}^ng_{ij}$. We say that $G$ has \textbf{optimal transport plans} if for any choice of $g_i$ and $x_i$ as above, the infimum in \eqref{optplans} is achieved.
\end{definition}
\noindent Note that the condition $g_{ii}=0$ is implied by $g_{ij}=-g_{ji}$, and thus becomes redundant, if $G$ has no elements $g$ of torsion $2$, i.e.\ such that $g+g=0$.

\medskip

\noindent It is clear that in case $G$ is proper, i.e.\ the closed balls $\mathbf B(0,r)$ are compact for all $r > 0$, then $G$ has optimal transport plans. As such, this is a rather weak condition on the group.
\begin{example}
For the case $G=\mathbb R$ with the Archimedean norm, we obtain the usual notion of optimal transport: due to the condition $g_1+\cdots+g_n=0$, up to reordering we may suppose that there exists $1\le k\le n$ such that $g_1\ge\cdots\ge g_k \geq 0 \geq g_{k+1}\ge\cdots\ge g_n$, and then the problem \eqref{optplans} becomes equivalent to that of transporting at minimal cost (where the transport cost is equal to the distance) the masses $|g_1|,\ldots,|g_k|$ situated at points $x_1,\ldots,x_k$ to masses $|g_{k+1}|,\ldots,|g_n|$ situated at the points $x_{k+1},\ldots, x_n$.
\end{example}
\subsection{Groups with nonbranching optimal transport plans}\label{branched_intro}
We next introduce the notion of branching, expressed in terms of group coefficients only. We point out the first basic example of branched transport as considered first by Gilbert in 1967 \cite{gilbert} and more recently formalized by Xia \cite{xia}, which appears for the case $G=\mathbb R$ with the norm $|x|_G \mathrel{\mathop:}= |x|^\alpha$, for $\alpha\in\ ]0,1[$. Then we have the strict subadditivity $|a+b|^\alpha < |a|^\alpha +|b|^\alpha$ for $a,b>0$, which is the fundamental reason why branching for optimal transport occurs (see the discussion in \cite{bcm}). The condition from Definition \ref{nbplans} below is precisely preventing this to occur, in the general case.
\begin{definition}
[Nonbranching optimal transport plans]
\label{nbplans}
Assume that $(G,|\cdot|)$ has optimal transport plans. We say that $(G,|\cdot|)$ has \textbf{nonbranching optimal transport plans} if for any finite collection $g_1,\ldots,g_n\in G$ with $\sum_{i=1}^n g_i=0$ there are $g_{ij}\in G$, $i,j = 1,\dots n$, with
\begin{equation}
	\tag{NBP}
	\label{nbp}
	\left\{
		\begin{array}{ll}
			g_{ij}=-g_{ji} & \text{ for all }i,j=0,\ldots,n, \\
			g_{ii}=0 &\text{ for all }i=1,\ldots,n,\\
			g_i=\sum_{j=1}^n g_{ij} &\text{ for all }i=1,\ldots,n, \\ 
			|g_i|=\sum_{j=1}^n|g_{ij}| &\text{ for all }i=1,\ldots,n.
		\end{array}
	\right.
\end{equation}
\end{definition}
\noindent Once we know that a group $G$ has nonbranching optimal transport plans, a next regularity condition to require is that the graph encoding how $G$-mass is transported along transport plans, does not have cycles. This kind of requirement turns out to generate interesting geometric conditions on $G$, and is the content of the next definition:
\begin{definition}[Acyclic nonbranching optimal transport plans]
\label{acyclic_def}
We say that the normed Abelian group $(G,|\cdot|)$ has \textbf{acyclic nonbranching optimal transport plans} if for any finite collection $g_1,\ldots,g_n\in G$ with $\sum_{i=1}^n g_i=0$ there are $g_{ij}\in G$, $i,j = 1,\dots n$, as in \eqref{nbp} such that the graph with vertices $V \mathrel{\mathop:}= \{g_i \ |\ 1\le i\le n\}$ and edges $\{\{g_i,g_j\}\in 2^V \ | \  g_{ij} \neq 0\}$ doesn't contain cycles.
\end{definition}
\noindent We will note in Section~\ref{class_cotpczt} that a first necessary condition on $G$ for having acyclic nonbranching optimal trasport plans, perhaps geometrically appealing in its own right, is the following:
\begin{definition}[collinear zero-mean triples]\label{collineardef}
Let $(G,|\cdot|)$ be a normed Abelian group. We say that $a,b,c\in G$ form a \textbf{zero-mean triple} if $a+b+c=0$. We say that a nontrivial triple is \textbf{collinear} if
\begin{equation}\label{collinear}
\text{One of }\ \ |a|+|b|=|c|,\ \ |a|+|c|=|b|,\ \ |b|+|c|=|a|\ \ \text{ holds.}
\end{equation}
We say that $(G,|\cdot|)$ has \textbf{collinear zero-mean triples} if 
\begin{equation}
	\tag{CZT}
	\label{czt}
	\mbox{ all zero-mean triples are collinear. }
\end{equation}
\end{definition}

\smallskip

\subsection{Branched transport is a minimal filling problem}
In fact, the minimization problems that are considered under the denomination of ``branched optimal transport'' are usually not formulated in the form of an optimal transport problem in which marginals are fixed and one minimizes over transport plans, but rather they are formulated exactly as a minimal filling problem. This link to the minimization among $G$-chains is also pointed out in \cite{xia}. Motivated by this fact, we introduce the nonbranching property defined in terms of minimal fillings.

\medskip

\noindent The spaces of rectifiable and flat $k$-dimensional chains in a metric space $X$ with coefficients in $G$ were defined by Fleming \cite{fleming} for $X = \mathbb R^n$ and extended by De Pauw and Hardt \cite{depauwhardt} to arbitrary metric spaces. A $0$-dimensional rectifiable chain in $\mathcal R_0(X;G)$ with finite support is simply a finite union of points $p_1,\dots,p_n$ in $X$ to each of which a coefficient $g_i$ in $G$ is associated. Such a chain $T$ is denoted by
\begin{equation*}
T = \sum_{i =1}^n g_i\curr{p_i} \ .
\end{equation*}
If $\gamma : [0,1] \to X$ is a Lipschitz path and $g\in G$, then a $1$-dimensional Lipschitz $G$-chain is given by $\gamma_\#(g\curr{0,1}) \in \mathcal L_1(X;G)$ and its mass is $\mathbf M(S) = |g| \operatorname{length}(\gamma)$ in case $\gamma$ is injective. See \cite{depauwhardt} for the precise definition of mass in this context. Moreover, any Lipschitz chain in $\mathcal L_1(X;G)$ is the finite sum of chains of this type. A rectifiable chain $S \in \mathcal R_1(X;G)$ is induced by a $G$-valued orientation $\mathbf g : A \to G$ on an oriented $1$-rectifiable Borel set $A \subset X$ such that the mass $\mathbf M(S) = \int_A |\mathbf g| \, d\mathcal H^1$ is finite, see \cite[\S 3]{depauwhardt} for more details. As a subset of rectifiable chains, a polyhedral chain $S \in \mathcal P_1(X;G)$ in a normed space $X$ is given by $S = \sum_{\sigma \in K^{(1)}} g_\sigma \curr \sigma$, where $g_\sigma \in G$ and $K \subset X$ is a finite oriented $1$-dimensional simplicial complex, see \cite[p.~1052]{depauwhardt2}. 

\medskip

\noindent The \textbf{filling problem} for $T \in \mathcal R_0(X;G)$ is the following minimization problem
\[
\rm{Fill}_{G,X}(T) \mathrel{\mathop:}= \inf\bigl\{ \mathbf M(C)\ |\ C\in\mathcal R_1(X;G),\ \partial C=T\bigr\} \ ,
\]
and as usual if $T$ is not a boundary, then the filling length is infinite. It should be noted that in a Lipschitz path connected metric space $X$, the chain $\sum_{i=1}^n g_i\curr{p_i}$ is a boundary of elements in $\mathcal R_1$ if and only if $\sum_{i=1}^n g_i = 0$. This can be proved by induction for Lipschitz chains using the identity \cite[Theorem~4.2.1]{depauwhardt}:
\[
\partial \gamma_\#(g\curr{0,1}) = g\curr{\gamma(1)} - g\curr{\gamma(0)} \ .
\]
For general rectifiable chains this follows by approximation, \cite[Theorem~4.3.4]{depauwhardt}. 

\medskip

\noindent First we state a definition that turns out to be equivalent to Definition~\ref{nbplans}, see Theorem \ref{equivalence_thm}.
\begin{definition}
\label{fillingsongeodesics}
We say that $(G,|\cdot|)$ has \textbf{nonbranching minimal fillings} if for all $g_1,\dots,g_n \in G$ such that $g_1+\cdots+g_n=0$ and all $x_1,\dots,x_n \in X$ in a geodesic metric space there is a $S \in \mathcal L_1(X;G)$ with
\begin{enumerate}
	\item $\partial S = T \mathrel{\mathop:}= \sum_{i=1}^n g_i \curr{x_i}$,
	\item $\rm{spt}(S) \subset \bigcup_{1 \leq i < j \leq n}[x_i,x_j]$, where $[x_i,x_j]$ is a geodesic segment connecting $x_i$ with $x_j$ in $X$,
	\item $\mathbf M(S) = \rm{Fill}_{G,X}(T)$.
\end{enumerate}
\end{definition}
\noindent Note that if the open ended geodesic segments $(x_i,x_j)$ in the definition above are pairwise disjoint, then the constancy theorem \cite[Theorem~6.4]{depauwhardt2} implies that
\[
S = \sum_{1 \leq i < j \leq n} g_{ij}\curr{x_j,x_i} \ ,
\]
for some $g_{ij} \in G$, $i,j = 1,\dots,n$, where we set $g_{ii} = 0$ and $g_{ji} = -g_{ij}$ for $j > i$. If we further assume that $x_i \neq x_j$ for $i \neq j$ the condition $\partial S = T$ implies that $g_i = \sum_{j=1}^n g_{ij}$ for all $i=1,\dots,n$.

\subsection{Main results}
As mentioned above, the two conditions of nonbranching (i.e.\ the one based on transport plans and the one based on fillings) are equivalent. We may interpret this by saying that that subadditivity phenomenon highlighted at the beginning of Subsection~\ref{branched_intro} is robust enough to pass to the case of general normed Abelian groups.
\begin{theorem}
\label{equivalence_thm}
Let $G$ be a normed Abelian group. The following are equivalent:
\begin{enumerate}
	\item $G$ has nonbranching optimal transport plans.
	\item $G$ has nonbranching minimal fillings.
\end{enumerate}
\end{theorem}
\noindent Our next step is to classify finitely generated groups that have nonbranching optimal transport or minimal fillings. First of all we note if $A$ and $B$ have nonbranching optimal transport plans, so does $A \times B$ with norm $|(a,b)| = \lambda|a|_A + \mu|b|_B$, where $\lambda,\mu > 0$ are arbitrary, see Lemma~\ref{directsum_lem}. This suggests that groups with nonbranching optimal transport plans are $\ell_1$-sums of elementary building blocks. Our next main result proves this, and completely classifies finitely generated normed Abelian groups which have nonbranching optimal transport plans (or minimal fillings, which is equivalent by Theorem~\ref{equivalence_thm}). 

\medskip 

\noindent Note that beyond finitely generated groups the class of normed Abelian groups is very large, and in particular contains the class of Banach spaces as a special subclass. Keeping this in mind, we also give a complete classification for the case where $(G,|\cdot|)$ is a finite-dimensional normed vector space, and we leave a more general classification of non-finitely generated groups with nonbranching optimal transport to a future work.
\begin{theorem}
\label{class_otp}
If $G$ is a normed Abelian group with nonbranching optimal transport plans, then the following two classification statements hold:
\begin{enumerate}
\item If $G$ is finitely generated, then $G$ is isometrically isomorphic to $\mathbb Z^k \times \mathbb Z_2^l$ with norm
\[
|(n_1,\dots,n_k,f_1,\dots,f_l)| = \sum_{i=1}^k \mu_i |n_i| + \sum_{j=1}^l \lambda_j |f_j| \ ,
\]
for some $\mu_i,\lambda_j > 0$.
\item If $(G,|\cdot|)$ can be endowed with a multiplication by scalars such that it becomes a finite dimensional normed vector space, then $G$ is isometrically isomorphic (as a normed vector space) to $\ell_1^n$ for some $n \geq 1$, where $\ell_1^n$ is the vector space $\mathbb R^n$ with norm $\|x\| = \sum_{i=1}^n|x_i|$.
\end{enumerate}
\end{theorem}
\noindent Our next result is a complete classification of groups with acyclic nonbranching transport plans. In this case, within the class of all normed Abelian groups, we find that only four groups satisfy the condition:
\begin{theorem}[Classification of groups with acyclic nonbranching optimal transport plans]
\label{class_cotp}
The following ones are the only complete normed Abelian groups that have acyclic nonbranching optimal transport plans, up to rescaling of their norm by a constant factor:
\begin{itemize}
\item $\mathbb R$ with the Archimedean norm,
\item $\mathbb Z$ with the Archimedean norm,
\item $\mathbb Z_2$,
\item $\mathbb Z_2\times \mathbb Z_2$ with norm satisfying $|(1,0)|=1, |(0,1)|=\alpha, |(1,1)|=1+\alpha$ for any choice of $\alpha \geq 1$.
\end{itemize}
\end{theorem}
\noindent The above theorem is based on the complete classification of groups with collinear zero-mean triples:
\begin{proposition}[Classification of groups with collinear zero-mean triples]
\label{class_czt}
The following ones are the only complete normed Abelian groups that have \eqref{czt}, up to rescaling of their norm by a constant factor:
\begin{itemize}
\item $\mathbb R$ with the Archimedean norm,
\item $\mathbb Z$ with the Archimedean norm,
\item $\mathbb Z_2$,
\item $\mathbb Z_4$ with norm satisfying $|1|=1, |2|=2$,
\item $\mathbb Z_2\times \mathbb Z_2$ with norm satisfying $|(1,0)|=1, |(0,1)|=\alpha, |(1,1)|=1+\alpha$ for any choice of $\alpha \geq 1$.
\end{itemize}
\end{proposition}
\noindent We then note that the groups with nonbranching optimal trasport plans as extracted in Theorem~\ref{class_otp} are endowed with a version of a global nonlinear duality, or, in more geometric terms, they have \textbf{calibrations}. This result is based on the corresponding result on the existence of calibrations for the minimum filling problem with coefficients in $\mathbb Z_2$ as obtained in \cite{pz} and on Kantorovich duality, for the cases of coefficients in $\mathbb R$ or $\mathbb Z$.
\begin{proposition}
\label{calibprop}
Let $G = \mathbb R^k \times \mathbb Z^l \times \mathbb Z_2^m$ with the $\ell_1$-norm as in Theorem~\ref{class_otp} i.e. 
\[
|(a_1,\dots,a_k,b_1,\dots,b_l,c_1,\ldots,c_m)| = \sum_{h=1}^k \lambda_h |a_h| + \sum_{i=1}^l \mu_i |b_i| + \sum_{j=1}^m \nu_j |c_j| \ ,
\]
for real numbers $\lambda_h,\mu_i,\nu_j>0$. Consider a chain $R = \sum_{i=1}^n g\curr {x_i} \in \mathcal R_0(X;G)$ such that $\sum_{i=1}^n g=0$ in a geodesic metric space $X$. Then
\begin{equation}
\label{filling_eq}
\rm{Fill}_{G,X}(R) = \max_{f_1,\dots,f_{k+l+m},T} \rm{Fill}_{G,T}\biggl(\sum_{j=1}^{k+l+m} f_{j\#} (\pi_j R) \biggr) \ ,
\end{equation}
where $T$ ranges over all finite geodesic trees and $1$-Lipschitz maps $f_i : X \to T$. Here $\pi_i : G \to G_i$ is the projection onto the $i$th factor and $\pi_i : \mathcal R_*(X;G) \to \mathcal R_*(X;G_i)$ is its induced map.
\end{proposition}
\noindent As a partial converse to this proposition we point out in Lemma~\ref{converse_lem} that any proper normed Abelian group that can be calibrated with maps into trees, needs to have nonbranching optimal transport plans.
\section{Proof of Theorem~\ref{equivalence_thm}}
\label{sec:eqchar}
\begin{proof}[Proof of Theorem~\ref{equivalence_thm}]
(1)$\Rightarrow$(2): Consider the case where $T = \sum_{i=1}^n g_i\curr{x_i} \in \mathcal R_0(\ell^m_\infty; G)$, where $\partial T = 0$ and $\ell^m_\infty$ is $\mathbb R^m$ equipped with the sup-norm. We also assume without loss of generality that $g_i\neq 0$ and that all the $x_i$ are different. 

\medskip

\noindent \textbf{Claim.} \textit{Let $S \in \mathcal P_1(\ell^m_\infty; G)$ with $\partial S = T$. For any $\epsilon > 0$ there exist $g_{ij} \in G$ for $1 \leq i < j \leq n$, such that $\sum_{i < j} |g_{ij}|d(x_i,x_j) \leq \mathbf M(S) + \epsilon$.}

\medskip

\noindent As $S$ is polyhedral of dimension $1$, it can be associated to an oriented graph. We may write $S = \sum_{e \in E(S)} g_e \curr{e}$, where $g_e \in G$ is such that $g_e \curr{e} \neq 0$, where $V(S)$ are the vertices and $E(S)$ are the oriented edges of a finite oriented graph. Identifying $\ell^m_\infty$ with a subspace of a larger dimensional space $\ell^{\mu}_\infty$ if necessary, we may find a polyhedral chain $S_0 \in \mathcal P_1(\ell^{\mu};G)$ such that:
\begin{enumerate}[(a)]
	\item $S_0 = \varphi_\# S$, where $\varphi : \rm{spt}(S_0) \to \ell^{\mu}_\infty$ is injective and affine on each edge of $E(S)$.
	\item $g_e\curr{\varphi(e)} \neq 0$ for all $e \in E(S)$.
	\item $d(y,\varphi(y)) \leq \epsilon$ for all $y \in \operatorname{spt}(S)$.
	\item $\mathbf M(S_0) \leq \mathbf M(S) + \epsilon$.
	\item If $v_1,v_2,v_3 \in V(S_0)$ are different, then $L(v_1,v_2) \cap L(v_1,v_3) = \{v_1\}$, where $L(v,w)$ is the line through $v$ and $w$.
	\item If $v_1,v_2,v_3,v_4 \in V(S_0)$ are different, then $L(v_1,v_2) \cap L(v_3,v_4) = \emptyset$.
\end{enumerate}
If $\tilde S \in \mathcal P_1(\ell^{\mu};G)$ has some underlying graph such that its vertices satisfy the above conditions (e) and (f), then we say that $\tilde S$ is in \textit{in general position}. We set $B \mathrel{\mathop:}= \rm{spt}(\partial S_0) = \{x_1',\dots,x_n'\}$, where $x_i' \mathrel{\mathop:}= \varphi(x_i)$ for all $i$. Successively $S_{\alpha+1}$ is constructed from $S_\alpha$ in case $V(S_\alpha) \neq B$ in such a way that $\mathbf M(S_{\alpha+1}) \leq \mathbf M(S_{\alpha})$ and $B \subset V(S_{\alpha+1}) \subsetneq V(S_\alpha) \subset V(S_0)$. Since $V(S_0)$ is a finite set, this process stops in a finite number of steps and we end up with a chain that is supported on straight line segments connecting points in $B$. For the iterative step, assume that $S_\alpha$ is already constructed and that $V(S_\alpha) \neq B$. Pick some $v \in V(S_\alpha)\setminus B$ and denote by $\sum_{i = 1}^l a_i \curr{v,v_i}$ the $G$-chain obtained by restricting $S_\alpha$ to the union of the segments that contain $v$. Since $G$ has nonbranching optimal transport plans there are $a_{ij}$, $i,j = 1,\dots,l$, satisfying \eqref{nbp}. Note then that by the general position assumption
\begin{align*}
\mathbf M\biggl(\sum_{i = 1}^l a_i \curr{v,v_i}\biggr) & = \sum_{i = 1}^l |a_i| d(v,v_i) = \sum_{i=1}^l\sum_{j=1}^l |a_{ij}|d(v,v_i) \\
 & = \sum_{1 \leq i < j \leq l} |a_{ij}|(d(v_j,v) + d(v,v_i)) \\
 & \geq \sum_{1 \leq i < j \leq l} |a_{ij}|d(v_j,v_i) \\
 & = \mathbf M\biggl(\sum_{1 \leq i < j \leq l} a_{ij} \curr{v_j,v_i}\biggr) \ .
\end{align*}
Now define
\[
S_{\alpha+1} \mathrel{\mathop:}= S_\alpha-\sum_{i = 1}^l a_i \curr{v,v_i}+\sum_{1\le i < j\le l} a_{ij} \curr{v_j,v_i} \ . 
\]
From \eqref{nbp} it follows that $\partial S_\alpha = \partial S_{\alpha+1}$ and since $V(S_\alpha) \subset V(S_0)$, the general position assumption on $V(S_0)$ implies that $V(S_{\alpha+1}) = V(S_\alpha)\setminus\{v\}$  (for the obvious choice of graph underlying $S_{\alpha+1}$).

\medskip

\noindent At the end of the iterative procedure we obtain a chain $S' = \sum_{i < j} g_{ij} \curr{x_j',x_i'}$ in $\mathcal P_1(\ell^{\mu}_\infty;G)$ with $\partial S' = \sum_{i=1}^n g_i\curr{x_i'}$ and hence $\sum_{j = 1}^n g_{ij} = g_i$. By construction, $\mathbf M(S') \leq \mathbf M(S) + \epsilon$ and $d(x_i',x_i) \leq \epsilon$ for all $i$. Set $D \mathrel{\mathop:}= \inf_{i\neq j} d(x_i,x_j)$ and assume further that $0 < \epsilon < \frac{1}{3}D$. Then $d(x_i',x_j') \geq \frac{1}{3} D$ and hence
\begin{align}
\nonumber
\sum_{1 \leq i < j \leq n} |g_{ij}|d(x_i,x_j) & \leq \sum_{1 \leq i < j \leq n} |g_{ij}|(d(x_i',x_j') + 2\epsilon) \\
\nonumber
 & = \mathbf M(S') + 2\epsilon\sum_{1 \leq i < j \leq n} |g_{ij}|  \\
\nonumber
 & \leq \mathbf M(S') + 6 \epsilon D^{-1} \sum_{1 \leq i < j \leq n} |g_{ij}|D \\
\label{massestimate}
 & \leq \mathbf M(S) + \epsilon(1 + 6 D^{-1} (\mathbf M(S) + \epsilon)) \ .
\end{align}
For $\epsilon>0$ small enough, our claim directly follows.

\medskip

\noindent Now we extend the claim to a general geodesic space $X$. Let $T \mathrel{\mathop:}= \sum_{i=1}^n g_i\curr{x_i} \in \mathcal R_0(X; G)$ for a general geodesic space $X$ and consider $S \in \mathcal R_1(X; G)$ such that $\partial S = T$. Again we can assume that $g_i \neq 0$ for all $1\le i\le n$ and that $x_i \neq x_j$ for different $i$ and $j$. Let $f : \operatorname{spt}(T) \to \ell^m_\infty$ be an isometric embedding (we could take $m = n$ for example). There exists a $1$-Lipschitz extension $\bar f : X \to \ell^m_\infty$ (this can be seen for example by applying the McShane-Whitney Lipschitz extension theorem to each coordinate function) and hence $\mathbf M(\bar f_\# S) \leq \mathbf M(S)$. Taking a $1$-Lipschitz projection of $\bar f_\# S$ onto a bounded set without changing the boundary $f_\# T$, we can assume that $\bar f_\# S$ has compact support. This allows to apply the polyhedral approximation result \cite[Theorem~4.2(D)]{depauw}, due to which, for any fixed $\epsilon>0$ we can find a polyhedral chain $P_\epsilon\in \mathcal P_1(\ell^m_\infty;G)$ with $\partial P_\epsilon = f_\# T$ and $\mathbf M(P_\epsilon) \leq \mathbf M(S) + \epsilon$. Thus by applying the claim to $P_\epsilon$ we find $g_{ij}$ as in \eqref{massestimate} such that
\begin{align}
\nonumber
\mathbf{M} \biggl(\sum_{1 \leq i < j \leq n} g_{ij}\curr{x_i,x_j}\biggr) & \leq \sum_{1 \leq i < j \leq n} \mathbf M(g_{ij}\curr{x_i,x_j}) \\
\nonumber
 & = \sum_{1 \leq i < j \leq n} \mathbf |g_{ij}|d(x_i,x_j) \\
\label{finaleq}
 & \leq \mathbf M(S) + \epsilon(2 + 6 D^{-1} (\mathbf M(S) + 2\epsilon)) \ .
\end{align}
We are now ready to conclude the proof of our first implication. $G$ has optimal transport plans by assumption and thus the infimum in \eqref{optplans} is achieved by some coefficients $g'_{ij} \in G$. The chain $R \mathrel{\mathop:}= \sum_{i < j} g'_{ij}\curr{x_j,x_i} \in \mathcal R_1(X;G)$ satisfies $\mathbf M(R) \leq \sum_{i < j} \mathbf |g'_{ij}|d(x_i,x_j)$. If we assume by contradiction that $\rm{Fill}_{G,X}(T) < \mathbf M(R)$, then by picking a suitable $S \in \mathcal R_1(X;G)$ and $\epsilon>0$, it follows from \eqref{finaleq} that there exist some $g_{ij} \in G$ such that
\[
\sum_{1 \leq i < j \leq n} \mathbf |g_{ij}|d(x_i,x_j) < \mathbf M(S) + \epsilon < \mathbf M(R) \leq \sum_{1 \leq i < j \leq n} \mathbf |g'_{ij}|d(x_i,x_j) \ ,
\]
contradicting the minimality of $g'_{ij}$.

\medskip

\noindent (2)$\Rightarrow$(1): Given an $n$-point metric space $(\{x_1,\dots,x_n\},d)$ and $g_1,\dots,g_n \in G$, let $X$ be the complete graph on the vertices $\{x_1,\dots,x_n\}$ equipped with the geodesic metric $d_X$ that agrees with $d$ on the vertex set. As in the discussion following Definition~\ref{fillingsongeodesics}, for each admissible solutions of \eqref{optplans} we can construct a minimal filling of $T = \sum_{i=1}^n g_i \curr{x_i} \in \mathcal R_0(X;G)$ and vice versa. Since $\mathbf M(S) = \sum_{i < j} |g_{ij}|d_X(x_j,x_i)$ for a filling $S = \sum_{i < j} g_{ij}\curr{x_j,x_i}$ of $T$, a mass minimal filling of $T$ is also a minimizer of \eqref{optplans}. Hence $G$ has optimal transport plans.

\medskip

\noindent Next we show that $G$ has nonbranching optimal transport plans. Let $g_1,\dots,g_n$ be elements in $G$ with $g_1+\dots,g_n=0$. Consider the infinite geodesic metric graph $(X,d)$ on the vertex set $V \mathrel{\mathop:}= \{x_1,\dots,x_n\} \sqcup \{c_1,c_2,\dots\}$ and with edges $E \mathrel{\mathop:}= \{\{x_i,x_j\}, \{x_i,c_k\}\ |\ i \neq j, k \in \mathbb N \}$. The length of the edges is given by $d(x_i,x_j) = 2$ if $i \neq j$ and $d(x_i,c_k) = 1 + \frac{1}{k}$. Consider the chain
\[
T \mathrel{\mathop:}= \sum_{i=1}^n g_i \curr {x_i} \in \mathcal R_0(X;G) \ .
\]
By Definition~\ref{fillingsongeodesics} and the discussion thereafter, there exist $g_{ij} \in G$ with $g_{ij} = -g_{ji}$, $g_{ii} = 0$, $g_i=\sum_{j=1}^ng_{ij}$ and
\begin{equation}
\label{upperbound}
\sum_{1 \leq i < j \leq n} |g_{ij}| d(x_i,x_j) = \mathbf M(S) \leq \mathbf M(C) \ ,
\end{equation}
where $S \mathrel{\mathop:}= \sum_{i < j} g_{ij} \curr{x_j,x_i}$ and $C \in \mathcal R_1(X;G)$ is an arbitrary filling of $T$. For each $k \in \mathbb N$ let $C_k \in \mathcal R_1(X;G)$ be the chain given by
\[
C_k \mathrel{\mathop:}= \sum_{i=1}^n g_i \curr{c_k,x_i} \ ,
\]
which obviously satisfies $\partial C_k = T$. We set $M \mathrel{\mathop:}= \sum_{i=1}^n |g_i|$. With the definition of $d$, \eqref{upperbound} and the triangle inequality we obtain for all $k \in \mathbb N$,
\begin{align*}
\sum_{i, j =1}^n |g_{ij}| & = \sum_{1\leq i < j\leq n} |g_{ij}|d(x_i,x_j) \leq \mathbf M(C_k) = \sum_{i=1}^n |g_i|(1 + \tfrac{1}{k}) \\
 & = \tfrac{1}{k}M + \sum_{i=1}^n |g_i| \ .
\end{align*}
Hence for all $i$,
\begin{equation}\label{mainbd}
\sum_{j=1}^n |g_{ij}| \leq |g_i| + \tfrac{1}{k}M \ .
\end{equation}
To justify this, note that if $A_i \leq B_i$ for all $i=1,\dots,n$ and $\sum_{i=1}^n B_i \leq \epsilon + \sum_{i=1}^n A_i$, then $B_i \leq \epsilon + A_i$ for all $i$. In the above situation we apply this for $\epsilon = \frac{1}{k}M$, $A_i = |g_i|$ and $B_i = \sum_j|g_{ij}|$. Since \eqref{mainbd} holds for all $k$, we obtain that $G$ has nonbranching optimal transport plans.
\end{proof}

\section{Proof of Theorem~\ref{class_otp}}
\subsection{Product lemma}
We first state the product lemma that was mentioned in the introduction.
\begin{lemma}
\label{directsum_lem}
Let $(A,|\cdot|_A)$ and $(B,|\cdot|_B)$ be two normed Abelian groups that have nonbranching optimal transport plans. Then the direct sum $(A \times B,|\cdot|)$ with norm given by $|(a,b)| \mathrel{\mathop:}= \lambda|a|_A + \mu|b|_B$, where $\lambda,\mu > 0$ are arbitrary, also have nonbranching optimal transport plans.
\end{lemma}
\begin{proof}
Let $(a_1,b_1),\dots,(a_n,b_n) \in A \times B$ be a finite collection with $\sum_i (a_i,b_i) = (0_A,0_B)$. By assumption there are $g_{ij} \in A$ and $h_{ij} \in B$ for $i,j \in \{1,\dots,n\}$ that satisfy \eqref{nbp}. Defining the norm as stated, the collection $(g_{ij},h_{ij}) \in A\times B$, $i,j \in \{1,\dots,n\}$, is easily seen to satisfy \eqref{nbp} for the data $(a_i,b_i)$, $i \in \{1,\dots,n\}$.

\medskip

\noindent With a similar argument we also obtain that $A\times B$ has optimal transport plans as in Definition~\ref{optitransp} if and only if both $A$ and $B$ have optimal transport plans. Hence $A \times B$ has nonbranching optimal transport plans since both $A$ and $B$ have.
\end{proof}

\subsection{Finitely generated groups}

An element $g$ in a normed Abelian group $G$ is called \textbf{indecomposable} if whenever $|h| + |g - h| = |g|$ for some $h \in G$, then $h = 0$ or $h=g$. Respectively, for all $h \in G \setminus \{0,g\}$ the inequality $|h| + |g-h| > |g|$ holds. Note that $0$ is indecomposable by this definition.

\medskip

\noindent For $h,g \in (G,|\cdot|)$ we write $h \perp \langle g\rangle$ if for all $n \in \mathbb Z$ the identity $|ng + h| = |ng| + |h|$ holds.
\begin{lemma}
\label{indecomplem}
If $G$ has nonbranching optimal transport plans and $g \in G \setminus \{0\}$ is indecomposable, then $2g = 0$ or
\[
n|g|=|ng| \ ,
\]
for all $n \in \mathbb N$. Moreover:
\begin{enumerate}
	\item If $|h| + |ng-h| = |ng|$ for some $n \in \mathbb Z$, then $h$ is a multiple of $g$.
	\item For all $h \in G$ there exists some $n \in \mathbb Z$ such that $|h - ng| = \inf_{m \in \mathbb Z}|h-mg|$.
	\item This minimizer $n$ in $(2)$ is unique if $\langle g \rangle = \mathbb Z$ and unique modulo $2$ if $\langle g \rangle = \mathbb Z_2$. Let us denote it by $n(h,g)$.
	\item $h - n(h,g)g \perp \langle g \rangle$.
\end{enumerate}
\end{lemma}
\begin{proof}
For some $n \geq 2$ consider the points $g_1 = \dots = g_n = g$ and $g_{n+1} = -ng$. By assumption there are $g_{i,j}$ for $i,j=1,\dots,n+1$ with the property:
\begin{equation*}
\left\{
\begin{array}{ll}
g_{i,j}=-g_{j,i} & \text{ for all }i,j=0,\ldots,n+1,\\
g_{i,i}=0 &\text{ for all }i=1,\ldots,n+1,\\ 
g_i=\sum_{j=1}^{n+1}g_{i,j} &\text{ for all }i=1,\ldots,n+1,\\ 
|g_i| = \sum_{j=1}^{n+1} |g_{i,j}| &\text{ for all }i=1,\ldots,n+1.
\end{array}
\right.
\end{equation*}
Because $g_i$ is indecomposable for all $i \in \{1,\dots,n\}$ there is exactly one $j \in \{1,\dots,n+1\} \setminus \{i\}$ with $g_{i,j} \neq 0$. Hence $g_{i,j} = g_i = g$ and this forces $j = n+1$. Otherwise, $j \leq n$ would imply $g = g_j = g_{j,i} = -g_{i,j} = -g$ contradicting that $2g \neq 0$. Thus
\[
|ng| = |g_{n+1}| = \sum_{j=1}^{n} |g_{n+1,j}| = \sum_{j=1}^{n} |-g| = n|g| \ ,
\]
and the result follows.

\medskip

\noindent (1): Next we show that if $|h| + |ng-h| = |ng|$, then $h$ is a multiple of $g$. The statement is clear if $n=1$ since $g$ is indecomposable. So we can assume that $2g\neq 0$ and hence $n|g| = |ng|$ for $n \geq 1$ by the first part. Let us write $h+h' = ng$ with $|h|+|h'|=|ng|$ and set $g_1 = \dots = g_n = g$, $g_{n+1} = -h$ and $g_{n+2} = -h'$. Then there are corresponding $g_{i,j}\in G$ as in \eqref{nbp}. The indecomposability of $g$ implies that for all $i=1,\dots, n$ there is exactly one $j$ with $g_{i,j}$ nonzero and equal $g$. Moreover, this $j$ is equal $n+1$ or $n+2$. Up to a reordering we have $g_{1,n+1} = \dots = g_{k,n+1} = g$, $g_{k+1,n+2} = \dots = g_{n,n+2} = g$ and $f = g_{n+1,n+2}$. It follows that $-h = -kg + f$ with $|h| = k|g| + |f|$ and similarly $-h' = (k-n)g - f$ with $|h'| = (n-k)|g| + |f|$. Hence
\[
n|g| = |ng| = |h|+|h'| = k|g| + |f| + (n-k)|g| + |f| = n|g| + 2|f| \ .
\]
This implies that $f = 0$ an therefore both $h$ and $h'$ are multiples of $g$.

\medskip

\noindent (2): Let $h \in G$ and assume first that $2g = 0$. Set $g_1 = g, g_2 = -h, g_3 = g+h$. \eqref{nbp} and the indecomposability of $g$ imply that $|g+h| = |h| + |g|$ or $|h| = |g| + |g+h|$. In this case it is also clear that (3) holds.

\medskip

\noindent Next assume that $n|g| = |ng|$ for all $n \in \mathbb N$. Let $m,n \in \mathbb Z$ and assume that $|ng - h| < |mg - h|$. Clearly, $m \neq n$.
Consider $g_1 = \dots = g_{s} = g$, $g_{s+1} = ng-h$ and $g_{s+2} = h-mg$, where $s \mathrel{\mathop:}= |m-n|$. We fix the corresponding $g_{i,j}\in G$ as in \eqref{nbp}. The indecomposability of $g$ implies that for all $i=1,\dots, s$ there is exactly one $j$ such that $g_{i,j}$ is nonzero and equal $g$. Moreover, this $j$ is equal $s+1$ or $s+2$. Up to a reordering we therefore have $g_{1,s+1} = \dots = g_{k,s+1} = g$, $g_{k+1,s+2} = \dots = g_{s,s+2} = g$ and $f = g_{s+1,s+2}$. It follows that $ng-h = g_{s+1} = -kg + f, |ng-h| = k|g| + |f|$ and similarly $h-mg = (k-s)g - f, |h-mg| = (s-k)|g| + |f|$. Hence
\[
k|g| + |f| = |ng-h| < |h-mg| = (s-k)|g| + |f| \ .
\]
This implies that $|mg - h|-|ng-h|$ is a positive multiple of $|g|$. Hence the values of $n \mapsto |ng-h|$ are discrete in $\mathbb R$ and therefore there exists $n \in \mathbb Z$ such that $|ng + h| = \inf_{m \in \mathbb Z}|mg+h|$.

\medskip

\noindent (3): Again we only have to treat the case where $n|g| = |ng|$ for all $n \in \mathbb N$. Let $h \in G$ and assume by contradiction that $n,m \in \mathbb Z$ are two different minimizers as in (2), i.e.\ $|ng-h| = |mg-h| = \inf_{k \in \mathbb Z}|kg-h|$. We may assume that $m-n > 0$ and consider $g_1 = \dots = g_{m-n} = g$, $g_{m-n+1} = ng-h$ and $g_{m-n+2} = h-mg$. We fix the corresponding $g_{i,j}\in G$ as in \eqref{nbp}. It follows as before that there is some $k \in \{1,\dots,m-n\}$ and $f \in G$ such that $ng-h = -kg + f$ with $|ng-h| = k|g| + |f|$ and similarly $h-mg = (k-m+n)g - f$ with $|h-mg| = (m-n-k)|g| + |f|$. This forces $k = 0$, otherwise
\[
|(n+k)g-h| = |f| < k|g| + |f| = |ng-h| \ ,
\]
contradicting the minimality of $|ng-h|$. Hence $m=n$.

\medskip

\noindent (4): In case $2g = 0$, this part is immediate from (2), so we only have to consider the case $\langle g \rangle = \mathbb Z$. Given some $h \in G$, the element $h' \mathrel{\mathop:}= h - n(h,g)g$ clearly satisfies $n(h',g) = 0$. We want to show that $|ng - h'| = |ng| + |h'|$ for all $n \in \mathbb Z$. This is obvious if $n=0$. Else consider $g_1 = \dots = g_{s} = g$, $g_{s+1} = -h'$ and $g_{s+2} = h'-ng$, where $s\mathrel{\mathop:}= |n|$. We fix the corresponding $g_{i,j}\in G$ as in \eqref{nbp}. Again since $g$ is indecomposable, for all $i=1,\dots,s$ there is exactly one $j$ with $g_{i,j}$ is nonzero and equal $g$. Moreover, this $j$ is $s+1$ or $s+2$. Up to a reordering, there is some $k \in \{1,\dots,s\}$ and $f \in G$ such that $g_{1,s+1} = \dots = g_{k,s+1} = g$, $g_{k+1,s+2} = \dots = g_{s,s+2} = g$ and $f = g_{s+1,s+2}$. This implies that $-h' = -kg + f$ with $|h'| = k|g| + |f|$ and similarly $h'-ng = (k-s)g - f$ with $|h'-ng| = (s-k)|g| + |f|$. Because $n(h',g) = 0$ it must hold that $k = 0$ and therefore $f = -h'$. Otherwise, $|h'-kg| = |f| < |h'|$. Thus
\[
|h'-ng| = (|n|-k)|g| + |f| = |n||g| + |h'| = |ng| + |h'| \ ,
\]
and applied to $h$,
\[
|h - n(h,g)g - ng| = |ng| + |h - n(h,g)g| \ .
\]
The above hold for all $n \in \mathbb Z$ and hence $h - n(h,g)g \perp \langle g \rangle$.
\end{proof}

\noindent Note that the points (1) to (4) of the lemma above are not used for the remaining discussion, but these properties may be useful for a more general classification of normed Abelian groups with nonbranching optimal transport plans.

\medskip

\noindent Next we make a simple observation about the subgroup generated by the indecomposable elements.

\begin{lemma}
\label{indecomplem2}
Let $(G,|\cdot|)$ be a normed Abelian group that has nonbranching optimal transport plans. If $g,h \in G$ are indecomposable and $\langle g \rangle \neq \langle h \rangle$, then
\[
|kg+lh| = |kg| + |lh| \ ,
\]
for all $k,l \in \mathbb Z$ and $\langle g,h\rangle \simeq \langle g\rangle \times \langle h\rangle$.
\end{lemma}
\begin{proof}
The statement is clear if $g=0$,$h=0$,$k=0$ or $l=0$, so we assume that this is not the case. From Lemma~\ref{indecomplem} it follows that the groups generated by $g$ and $h$ are isomorphic to $\mathbb Z$ or $\mathbb Z_2$. Clearly, $f$ is indecomposable if and only if $-f$ is indecomposable. So by replacing $g$ or $h$ by its inverse we can assume that $k,l \geq 1$. In case $\langle g\rangle$ or $\langle h\rangle$ is isomorphic to $\mathbb Z_2$ the corresponding integer is assumed to be $1$. Consider $g_1=\dots=g_{k} = g$, $g_{k+1}=\dots=g_{k+l} = h$ and $g_{k+l+1} = -kg-lh$. We fix the corresponding $g_{i,j}\in G$ as in \eqref{nbp}. Since $g$ and $h$ are indecomposable, for all $i=1,\dots, k$ there is exactly one $j$ for which $g_{i,j}$ is nonzero and equal $g$ and similarly for $i=k+1,\dots,k+l$. Moreover, this $j$ is equal $k+l+1$. We therefore have $g_{1,k+l+1} = \dots = g_{k,k+l+1} = g$, $g_{k+1,k+l+2} = \dots = g_{k+l,k+l+1} = h$, which implies that $|kg+lh| = k|g| + l|h|$ and hence $|kg+lh| = |kg| + |lh|$ by the triangle inequality.

\medskip

\noindent From $|kg+lh| = k|g| + l|h|$ it follows that if $kg + lh = 0$, then $k=0$ and $l=0$. Hence $\langle g\rangle \cap \langle h\rangle = \emptyset$ and hence $\langle g,h\rangle \simeq \langle g\rangle \times \langle h\rangle$.
\end{proof}
\noindent For a normed Abelian group $(G,|\cdot|)$ we denote by $I_G \subset G$ a choice of a subset such that for any nonzero indecomposable element $g$ either $g$ or $-g$ is in $I_G$.

\begin{lemma}
\label{indecompcor}
Let $(G,|\cdot|)$ be a normed Abelian group that has nonbranching optimal transport plans. Then
\[
\left(\langle I_G\rangle,|\cdot|\right)\text{ is isometrically isomorphic to }\bigotimes_{g\in I_G}^{\ell_1}\left(\langle g\rangle, |\cdot|\right),
\]
where $\bigotimes^{\ell_1}$ represents the $\ell_1$-direct product.
\end{lemma}
\begin{proof}
The proof of Lemma~\ref{indecomplem2} can easily be generalized to include any finite collection of indecomposable elements.
\end{proof}

\subsection{Finite dimensional normed spaces}
It follows directly from Lemma~\ref{directsum_lem} that $(\mathbb R^n,\|\cdot\|_1)$ has nonbranching optimal transport plans. In this subsection we want to show the converse.

\medskip

\noindent An extreme point in a convex set $C \subset X$ in some Banach space $X$ is a point $p \in C$ that can't be written as $\lambda p_1 + (1-\lambda) p_2$, where $\lambda \in (0,1)$ and $p_1,p_2 \in C \setminus \{p\}$.
\begin{lemma}
\label{extreme_lem}
Let $p$ be an extreme point of the closed unit ball $\mathbf B_X(0,1)$ in some normed space $(X,\|\cdot\|)$. Then the equations $p = p_1 + p_2$ and $\|p\| = \|p_1\| + \|p_2\|$ can only hold if both $p_1$ and $p_2$ are multiples of $p$.
\end{lemma}
\begin{proof}
As an extreme point, $p$ has unit norm. The conclusion of the lemma is clear if $\|p_1\|=0$ or $\|p_2\|=0$. If this is not the case, then $p = \lambda \frac{p_1}{\|p_1\|} + (1-\lambda) \frac{p_2}{\|p_2\|}$, where $\lambda = \|p_1\| = \|p\|-\|p_2\| = 1 - \|p_2\|$. By assumption $\lambda \neq 0$ and therefore $\frac{p_1}{\|p_1\|} = p$ or $\frac{p_2}{\|p_2\|} = p$ since $p$ is extreme. But then 
the equation $p = \lambda \frac{p_1}{\|p_1\|} + (1-\lambda) \frac{p_2}{\|p_2\|}$ implies that $\frac{p_1}{\|p_1\|} = \frac{p_2}{\|p_2\|} = p$.
\end{proof}
\begin{lemma}
\label{l1decomp_lem1}
Assume that the normed space $(X,\|\cdot\|)$ has nonbranching optimal transport plans. If $p_1,\dots,p_n \in \mathbf B_X(0,1)$ are linearly independent extreme points and $\lambda_1,\dots,\lambda_n \in \mathbb R \setminus \{0\}$, then
\[
\biggl\|\sum_{i=1}^n \lambda_i p_i \biggr\| = \sum_{i=1}^n |\lambda_i| \|p_i\| \ .
\]
\end{lemma}
\begin{proof}
Consider the points $g_1 = \lambda_1 p_1, \dots, g_n = \lambda_n p_n$, $g_{n+1} = -\sum_{i=1}^n \lambda_i p_i$. By assumption there are $g_{i,j} \in G$ that satisfy \eqref{nbp}. With Lemma~\ref{extreme_lem} we conclude that for each $i$, all the vectors $g_{i,1},\dots, g_{i,n}$ are multiples of $p_i$. Since all the $p_i$'s are linearly independent this implies that $g_{i,j} = 0$ for all $i,j \in \{1,\dots,n\}$. Hence $g_{i,n+1} = \lambda_ip_i$ and further
\[
\biggl\|\sum_{i=1}^n \lambda_i p_i \biggr\| = \|g_{n+1}\| = \sum_{i=1}^n\|g_{i,n+1}\| = \sum_{i=1}^n|\lambda_i|\|p_{i}\| \ ,
\]
as claimed.
\end{proof}
\begin{lemma}
\label{l1decomp_lem2}
Assume that $X$ is a normed space of dimension $n$ that has nonbranching optimal transport plans, then $X$ is linearly isometric to $\ell^n_1$.
\end{lemma}
\begin{proof}
According to Lemma~\ref{l1decomp_lem1} the only thing that needs to be shown is that $\mathbf B_X(0,1)$ has at least $n$ linearly independent extreme points. But this is a simple consequence of the Krein-Milman theorem.
\end{proof}
\noindent Together with Lemma~\ref{indecompcor} this proves Theorem~\ref{class_otp}.

\section{Proof of Theorem~\ref{class_cotp} and of Proposition~\ref{class_czt}}
\label{class_cotpczt}
\noindent We recall that Theorem \ref{class_cotp} concerns the classification of all normed Abelian groups $(G,|\cdot|)$ that have acyclic nonbranching optimal transport plans. We first show that the groups mentioned in the statement of Theorem~\ref{class_cotp} indeed satisfy the properties stated in Definition~\ref{acyclic_def}.
\begin{lemma}
\label{check_groups}
All the groups listed in Theorem~\ref{class_cotp} have acyclic nonbranching optimal transport plans.
\end{lemma}
\begin{proof}
The proof for $\mathbb Z_2$ is straightforward since a transport plan for $g_1= \dots= g_{2n} = 1 \in \mathbb Z_2$ is obtained for example by the pairing $g_{2i-1,2i} = g_{2i,2i-1} = 1$ for $i=1,\dots,n$ and $g_{ij} = 0$ otherwise. Being a disjoint union of edges, the corresponding graph is automatically cycle-free.

\medskip
\noindent The argument for $\mathbb Z_2\times \mathbb Z_2$ is similar. Consider a collection $a_1,\dots,a_i = (1,1)$, $b_1,\dots,b_j = (1,0)$, $c_1,\dots,c_k = (0,1)$ of elements in $\mathbb Z_2\times \mathbb Z_2$ with total sum equal to zero. We may construct a nonbranching optimal transport plan $\{g_{ij}\}$ that verifies \eqref{nbp} by pairing off elements inside each set $\{a_1,\dots,a_i\}$, $\{b_1,\dots,b_j\}$ and $\{c_1,\dots,c_k\}$ separately. We are then left with the leftover cases where $i,j,k \in \{0,1\}$. The only nontrivial situation occurs when $i=j=k=1$ with $a = (1,1), b=(1,0)$ and $c=(0,1)$. In this case the transport plan given by $g_{ab} = (1,0)$, $g_{ac}=(0,1)$ and $g_{bc}=(0,0)$ doesn't produce a cycle.
\medskip
Finally we assume that $G = \mathbb R$. The proof for $G = \mathbb Z$ is similar. We prove the statement by induction on the number of elements $g_1,\dots,g_n \in \mathbb R$ with $g_1 + \dots + g_n = 0$. In case $n = 2$, the problem is trivial. Assume that the statement holds for $n-1 \geq 2$. Assume that the points $g_1,\dots,g_n \in \mathbb R$ with $g_1 + \dots + g_n = 0$ are ordered in such a way that $g_{i} \geq g_{i+1}$. Since all of the points sum to zero, it holds that $g_1 \geq 0 \geq g_n$ and by symmetry we can assume without loss of generality that $g_1 \geq |g_n|$. By induction, we can solve the problem for the points $h_1 \mathrel{\mathop:}= g_1 + g_n, h_2 \mathrel{\mathop:}= g_2,\dots, h_{n-1} \mathrel{\mathop:}= g_{n-1}$ to obtain a acyclic nonbranching optimal transport plan $g_{ij}$, $i=1,\dots,n-1$. By adding to these elements the elements $-g_{n1} = g_{1n} \mathrel{\mathop:}= -g_n$ and $g_{jn}=g_{nj} \mathrel{\mathop:}= 0$ for $2\le j\le n$ we obtain an acyclic nonbranching optimal transport plan for the original problem.
\end{proof}
\noindent The classification of these groups is simplified by first considering the case of only three elements in $G$. Indeed, in order to obtain the groups in Theorem~\ref{class_cotp} we only have to consider the cases $n=3$ and $n=4$ in Definition~\ref{acyclic_def}.
\begin{lemma}
\label{collinearimplication}
If $G$ has acyclic nonbranching optimal transport plans, then $G$ has collinear zero-mean triples \eqref{czt}.
\end{lemma}
\begin{proof}
Let $g_1,g_2,g_3 \in G$ with $g_1+g_2+g_3 = 0$ and $g_{ij}$ as in Definition~\ref{acyclic_def}. Since the graph associated with $g_{ij}$ doesn't contain a cylce there is some $g_{ij}$, say $g_{23}$, such that $g_{23} = 0$. Then
\begin{equation*}
|g_1| = |g_{12}| + |g_{13}| = |g_{2} - g_{23}| + |g_{3} + g_{23}| = |g_2| + |g_3| \ .
\end{equation*}
This is precisely what we want.
\end{proof}
\noindent In order to classify all the groups with acyclic nonbranching optimal transport plans we first classify the groups with collinear zero-mean triples.

\subsection{Torsion groups with collinear zero-mean triples}
We recall that a group $G$ is a \textbf{torsion group} if for all $g\in G$ there exists a natural number $n$ such that summing $g$ to itself $n$ times we obtain $0_G$, i.e.\ $ng=0_G$.
\begin{proposition}[Classification of torsion groups with \eqref{czt}]
\label{classtorsion}
The following ones are the only normed Abelian torsion groups that have \eqref{czt}, up to rescaling of their norm by a constant factor:
\begin{itemize}
\item $\mathbb Z_2$,
\item $\mathbb Z_4$ with norm satisfying $|1|=1, |2|=2$,
\item $\mathbb Z_2\times \mathbb Z_2$ with norm satisfying $|(1,0)|=1, |(0,1)|=\alpha, |(1,1)|=1+\alpha$ for any choice of $\alpha \geq 1$.
\end{itemize}
\end{proposition}

\begin{proof}
\par For the groups $\mathbb Z_2,\mathbb Z_4,\mathbb Z_2\times\mathbb Z_2$ the determination of the norms satisfying \eqref{czt} follows directly by enumerating the zero-mean triples in every case.

\medskip

\noindent Recall that any finitely generated Abelian torsion group has a direct product decomposition in which the factors are $\mathbb Z_{p^m}$ where $p$ is a prime number and $m$ is an integer, thus we just restrict to discussing such factors.

\medskip

\noindent Suppose some $\mathbb Z_n$ with $n$ odd has some norm $|\cdot|$ with \eqref{czt}. Let $a \in \mathbb Z_n$ be an element that maximizes $|a|$ and consider the triple $a$, $a$, $-2a$ in $\mathbb Z_n$. Since $n$ is odd and $a \neq 0$ it must be the case that $|-2a|\neq 0$ and hence none of the inequalities $|a| + |a| = |-2a|$ and $|a| + |-2a| = |a|$ can be satisfied by the maximality of $|a|$.

\medskip

\noindent We next exclude the factors $\mathbb Z_n$ for $n=2^{m+1}, m\ge2$. In this case, using the \eqref{czt} property we find that $|2|=|1|+|1|$ is the only possible collinearity formula for the triple $1,1,-2$, and similarly $|2^k|=|2^{k-1}|+|2^{k-1}|$ is the only possible collinearity formula for $2^{k-1}, 2^{k-1}, -2^k$ and $k=1,\dots,m$. By induction we find $|2^m|=2^m|1|$. But as $2(2^m-1)\equiv -2\ (\rm{mod}\ 2^{m+1})$ we also similarly find $2|2^m-1|=|2|=2|1|$ thus the zero-mean triple $2^m, 2^m-1,1$ has norms proportional to $2^m,1,1$, and this contradicts the triangular inequality for $m\ge 2$.

\medskip

\noindent In order to exclude $G = \mathbb Z_2\times\mathbb Z_2\times \mathbb Z_2$, consider $g_1,g_2,g_3 \in G\setminus \{0\}$ such that $|g_1| = \min \{|g|\ | \ g \in G\setminus\{0\}\}$, $|g_2| = \min \{|g|\ | \ g \in G\setminus\langle g_1\rangle\}$ and $|g_3| = \min \{|g|\ | \ g \in G\setminus \langle g_1,g_2\rangle\}$. The elements $g_1,g_2,g_3$ are indecomposable in the sense that whenever $g_i,g,h$ is a zero mean triple, then $|g_i|+|h| = |g|$ or $|g_i|+|g|=|h|$. Since the elements $g_1,g_2,g_3$ generate $G$ we can express $G$ as the product $\langle g_1 \rangle \times \langle g_2 \rangle \times \langle g_3 \rangle$ and identify $g_1=(1,0,0)$, $g_2=(0,1,0)$, $g_3=(0,0,1)$. Then the norm of an element $(x,y,z)$, where we chose $\mathbb Z$-representatives $x,y,z\in\{0,1\}$, must be given by $|(x,y,z)|=\alpha x+ \beta y + \gamma c$ for some $0 < \alpha \leq \beta \leq \gamma$. For the collinear triple $(1,1,0), (1,0,1), (0,1,1)$ in particular we have the norms $\alpha+\beta$, $\alpha+\gamma$, $\beta+\gamma$ and we find that the only possible collinearity formula is $2\alpha + \beta+\gamma=\beta+\gamma$, and thus is false. This provides a contradiction to the existence of a norm on $\mathbb Z_2\times \mathbb Z_2 \times \mathbb Z_2$ that satisfies \eqref{czt}.

\medskip

\noindent As a consequence of the above, the only possible factors in the direct decomposition of $G$ which are still allowed are $\mathbb Z_2$, $\mathbb Z_4$, and we know that $\mathbb Z_2$ can appear at most twice in this product. 

\medskip

\noindent It thus remains to exclude the appearance of $\mathbb Z_2\times \mathbb Z_4$ and of $\mathbb Z_4\times \mathbb Z_4$. Suppose that a norm on the group $\mathbb Z_2\times\mathbb Z_4$ had \eqref{czt} and set $\alpha \mathrel{\mathop:}= |(0,2)|$. As the triple $(1,1)$, $(1,1)$, $(0,2)$ has zero mean, we must have by collinearity that $|(1,1)|=\frac{1}{2}|(0,2)|=\frac{1}{2}\alpha$. Similarly $|(0,1)|=\frac{1}{2}\alpha$. Then the triple $(0,1)$, $(1,1)$, $(1,2)$ is of zero mean and thus $|(1,2)| = \frac{1}{2}\alpha + \frac{1}{2}\alpha = \alpha$. But also $(1,0)$, $(0,2)$, $(1,2)$ is of zero mean with $|(0,2)|=|(1,2)|=\alpha$, hence $|(1,0)| = 2\alpha$. Finally the zero mean triple $(1,0)$, $(1,1)$, $(0,-1)$ implies the collinearity of the numbers $2\alpha$, $\frac{1}{2}\alpha$, $\frac{1}{2}\alpha$, a contradiction. Therefore $\mathbb Z_2\times\mathbb Z_4$ has no norm satisfying \eqref{czt}. As $\mathbb Z_4\times \mathbb Z_4$ has a $\mathbb Z_2\times\mathbb Z_4$-subgroup as well, it also has no norm satisfying \eqref{czt}.
\end{proof}

\subsection{Torsion-free groups with collinear zero-mean triples}
We say that two normed Abelian groups $(G,|\cdot|_G)$ and $(H,|\cdot|_H)$ are equivalent $(G,|\cdot|_G) \simeq (H,|\cdot|_H)$ if there is a $\lambda > 0$ and a group isomorphism $\varphi : G \to H$ such that $|g|_G = \lambda |\varphi(g)|_H$ for all $g \in G$. Also recall that a group is \textbf{torsion-free} if there are no $g\in G\setminus \{0_G\}$ and $n \in \mathbb N$ with $ng=0_G$. In this subsection we want to prove the following proposition. 

\begin{proposition}[Classification of torsion-free groups with \eqref{czt}]
\label{classtorsionfree}
Let $(G,|\cdot|_G)$ be a complete torsion-free normed Abelian group that satisfies \eqref{czt}. Then either $G \simeq \mathbb Z$ or $G \simeq \mathbb R$.
\end{proposition}
\noindent If not stated otherwise, for the remainder of this subsection $(G,|\cdot|)$ denotes a torsion-free normed Abelian group that has \eqref{czt}. As a consequence of the fact that $G$ is torsion-free, for all $g \in G \setminus \{0_G\}$ the subgroup $\langle g \rangle < G$ is isomorphic to $\mathbb Z$. We prove first the following:

\begin{lemma}
\label{normonz}
Assume that $G=\langle g\rangle$ is an infinite cyclic group and $|\cdot|_G$ is a norm on it. Then $(G,|\cdot|_G)$ has \eqref{czt} if and only if for all $n \in \mathbb N$ it holds that
\begin{equation}
\label{cyclicsubgroups2}
|ng|_G = n|g|_G \ ,
\end{equation}
i.e.\ if and only if $(G,|\cdot|_G)$ is isomorphic to $\mathbb Z$ with its Archimedean norm.
\end{lemma}

\begin{proof}
The fact that the group $\mathbb Z$ with the usual norm has only zero-mean triples which are collinear follows by noting that $a=\pm|a|$ in these cases, and that the zero mean conditions $a+b+c=0$, for $a,b,c\neq0$ imply that $a,b,c$ don't have all the same sign.

\medskip

\noindent If $|\cdot|_G$ is a norm on $G\simeq\mathbb Z$ for which \eqref{czt} is true and if we have the normalization $|g|_G=1$, then by using iteratively the condition that the triples $-g$, $-ng$, $(n+1)g$ for $n=1,2,\ldots, $ are collinear we successively find $|2g|_G=2$ and for $n\ge 3$ we have $ |\pm(n+1)g|_G=n\pm 1$. For $n > 1$ the triple $-2g$, $-(n+1)g$, $(n+1)g$ (and the induction hypothesis $|-(n - 1)g|_G = n-1$) shows that $|\pm(n+1)g|_G=n-1$ is not allowed, thus the only norm with \eqref{czt} is the one satisfying \eqref{cyclicsubgroups2} as desired. 
\end{proof}
\noindent For $a,b \in G \setminus \{0_G\}$ we write $a \sim b$ if
\[
|a-b| < |a| + |b| \ .
\]
It is clear that $\sim$ is reflexive and symmetric. Next we want to establish that this is indeed an equivalence relation. Note that as a consequence of the triangle inequality, $a\sim b$ if and only if $|a-b| \neq |a|+|b|$.
\begin{lemma}
\label{equivalenceproperties}
Let $a,b,c \in G \setminus \{0_G\}$ and $m,n \in \mathbb N$. Then the following properties hold:
\begin{enumerate}
	\item $a \sim b$ or $a \sim -b$ is satisfied.
	\item $a \sim b$ if and only if $ma \sim nb$.
	\item If $a \sim b$ and $b \sim c$, then $a \sim c$.
	\item Only one of $a \sim b$ or $a \sim -b$ is satisfied.
	\item If $a \sim b$, then $a+b \sim a$.
\end{enumerate}
\end{lemma}

\begin{proof}
We prove (1) by contradiction. We can assume that $a \neq b$, otherwise the statement is trivial. If $a \sim b$ and $a \sim -b$ doesn't hold, the equalities $|a-b| = |a| + |b|$ and $|a+b| = |a| + |b|$ holds and hence $|a-b|=|a+b|$. Consider the zero-mean triple $a-b$, $a+b$, $-2a$. Since $|-2a|=|2a| = 2|a| \neq 0$ by Lemma~\ref{normonz}, none of the equalities $|a-b| + |2a| = |a+b|$ and $|a+b| + |2a| = |a-b|$ hold. Since $G$ has \eqref{czt} we must therefore have that $|a+b| + |a-b| = 2|a|$. But again by our initial assumption we have $|a+b| + |a-b| = 2|a| + 2|b| > 2|a|$, as $b\neq 0_G$.

\medskip

\noindent In order prove (2) we first show that $a \sim b$ implies $ma \sim b$ for $m \geq 1$. Note that by Lemma~\ref{normonz} it holds that
\begin{align*}
|ma - b| & \leq |(m-1)a| + |a-b| < (m-1)|a| + |a| + |b| = m|a| + |b| \\
 & = |ma| + |b| \ .
\end{align*}
By applying the same reasoning to the elements $ma$ and $b$ and some multiplier $n \geq 1$, we find that $a \sim b$ implies $ma \sim nb$ for $m,n \geq 1$. On the other side if $na \sim mb$, then $nma \sim nmb$ by the first step. And again by Lemma~\ref{normonz} setting $k = nm \geq 1$,
\begin{align*}
k|a - b| & = |k(a - b)| =|ka - kb| < |ka| + |kb| = k(|a| + |b|) \ .
\end{align*}
Dividing both sides by $k$ shows (2).

\medskip

\noindent Next, for proving (3) we assume that $a\sim b$ and $b\sim c$ and we desire to prove that $a\sim c$. Because $a \sim b$, it holds that $na \sim b$, i.e.\ $|na-b| < |na| + |b|$, for all $n \geq 1$ by (2). Since $G$ has \eqref{czt} and $-na$, $b$, $na-b$ is a zero-mean triple, one of the equations $|na-b| + |na| = |b|$ or $|na-b| + |b| = |na|$ must hold. For $n>|b||a|^{-1}$ the first equation can't hold because by the triangle inequality and by Lemma~\ref{normonz} we have
\[
|b| < 2n|a| - |b| = |na| - |b| + |na| \leq |na-b| + |na| \ .
\]
This similarly applies to the pair $(c,b)$ in place of $(a,b)$. So if $n$ is large enough we therefore have $|na-b| + |b| = |na|$ and $|nc-b| + |b| = |nc|$. Adding these two equalities we obtain by the triangle inequality and Lemma~\ref{normonz} that
\begin{align*}
n|c| + n|a| & = 2|b|+|na-b|+|nc-b| \geq 2|b| + |na-nc| = 2|b| + n|a-c| \ .
\end{align*}
Dividing by $n$ (where $n$ is chosen such that $n > |b|\max\{|a|^{-1},|c|^{-1}\}$) we obtain that $|a| + |c| > |a-c|$ and therefore $a \sim c$.

\medskip

\noindent (4) is a consequence of (3). Indeed, assume by contradiction that both relations $a\sim b$ and $a\sim -b$ hold. Then it follows from (3) that $b \sim -b$, i.e.\ $|2b| < |b| + |b|$. But this is not possible since $|2b| = 2|b|$ by Lemma~\ref{normonz}.

\medskip

\noindent Finally we show (5). If $a\sim b$, then $|a-b| < |a| + |b|$, and by using twice the triangle inequality we find
\[
|a+b| + |a| \geq 2|a| + |b| > |a-b| + |a| \geq |b| = |(a+b)-a| \ ,
\]
which implies that $a+b\sim a$, as desired.
\end{proof}
\noindent The above lemma shows that $\sim$ is an equivalence relation and (1) together with (5) show that $G \setminus \{0_G\}$ is the disjoint union of exactly two equivalence classes ($\{0_G\}$ being the third). Fixing some arbitrary $g_+ \in G\setminus \{0_G\}$, these two classes are given by $G_+ \mathrel{\mathop:}= \{g\in G \ | \ g\sim g_+\}$ and $G_- \mathrel{\mathop:}= -G_+$. Consider the map $\varphi : G \to \mathbb R$ defined by
\[
\varphi(g) \mathrel{\mathop:}=
\left\{
	\begin{array}{rll}
		 |g|_G & \mbox{if} & g \in G_+ \ , \\
		-|g|_G & \mbox{if} & g \in G_- \ ,\\
		0 & \mbox{if} & g = 0_G \ .
	\end{array}
\right.
\]
\begin{lemma}\label{isometrytoreals}
The map $\varphi: (G, |\cdot|_G)\to (\mathbb R,|\cdot|)$ is an isometric embedding and a group homomorphism.
\end{lemma}
\begin{proof}
We will use the notation $|\cdot|$ for the norm $|\cdot|_G$ in the proof, as the only time when the norm on $\mathbb R$ intervenes is in the last sentence of the proof. In order to show that $\varphi$ is a homomorphism we need to show that $\varphi(a+b) = \varphi(a) + \varphi(b)$ and $-\varphi(a) = \varphi(-a)$ for all $a,b \in G$. The second equality is obvious because of Lemma~\ref{equivalenceproperties}(4) the relation $a \sim -a$ never holds for $a \neq 0_G$. To prove the first equality we consider only the nontrivial case $a,b\in G\setminus\{0_G\}$. In view of Lemma~\ref{equivalenceproperties}(5), if $a\sim b$, then $a+b\sim a\sim b$, and if $a\sim -b$ as well as $a+b\sim a$, then $a+b\sim a\sim -b$. So up to interchanging $a$ and $b$, either $a+b\sim a\sim b$ or $a+b\sim a\sim -b$. We claim:
\begin{enumerate}
	\item $|a+b| = |a| + |b|$ in case $a+b\sim a\sim b$,
	\item $|a+b| = |a|-|b|$ in case $a+b\sim a\sim-b$.
\end{enumerate}
\noindent \textbf{Proof of (1)}: Translating $a+b\sim a$ and $a+b\sim b$ we have $|b| = |(a+b) - a| < |a+b| + |a|$ and similarly $|a| < |a+b| + |b|$. Since $G$ has \eqref{czt} and considering the zero-mean triple $a+b$, $-a$, $-b$, one of the following equalities has to hold: $|b| = |a+b| + |a|$, $|a| = |a+b| + |b|$, $|a+b|=|a|+|b|$. Since the first two are excluded we have $|a+b| = |a|+|b|$.

\medskip

\noindent \textbf{Proof of (2)}: As above we obtain $|b| < |a| + |a+b|$ from $a+b \sim a$ and $|a+b| = |a-(-b)|<|a|+|b|$ from $a \sim - b$. Since $G$ has \eqref{czt} and again considering the triple $a+b$, $-a$, $-b$, we get that $|a| = |b| + |a+b|$.

\medskip

\noindent This shows that $\varphi$ is a homomorphism. It also follows that $\varphi$ is an isometric embedding in the sense of metric spaces because $|\varphi(a)| = |a|_G$ by the definition of $\varphi$, and $|\varphi(a)-\varphi(b)|=|\varphi(a-b)|=|a-b|_G$ because $\varphi$ is a homomorphism.
\end{proof}
\noindent Recall that a normed Abelian group $(G,|\cdot|)$ is discrete if
\[
\inf\{|g|\ |\ g \in G \setminus \{0\} \} > 0 \ .
\]
Note that if $G$ is a discrete complete subgroup of $\mathbb R$ then $G$ is isomorphic to $\mathbb Z$. Indeed if $G \subset \mathbb R$ is complete, it is in particular a closed set and there exists an element $g\neq 0$ of smallest norm. If $G\neq \langle g\rangle$, then there would exist some $n\in\mathbb Z$ and $g'\in G \setminus \langle g \rangle$ with $|ng-g'| < |g|$ contradict the minimality of $|g|$.
\begin{proof}[Proof of Proposition~\ref{classtorsionfree}]
Because of Lemma~\ref{isometrytoreals} there is an isometric (in particular injective) homomorphism $\varphi : G \to \mathbb R$. If $G$ is not discrete, then $0 \in G$ is an accumulation point and hence the image $\varphi(G)$ is dense in $\mathbb R$. Since $G$ is complete, so is $\varphi(G)$ and $\varphi$ is therefore surjective. This shows that $\varphi$ is an isometric isomorphism. If $G$ is discrete, then the image of $\varphi$ must also be discrete and thus $G$ is isomorphic to $\mathbb Z$.
\end{proof}
\noindent The following example shows that the completeness assumption in Theorem~\ref{classtorsionfree} is necessary.
\begin{example}
Let $r,s \in \mathbb R\setminus\{0\}$ be such that $\frac{r}{s}$ is irrational and define $f : \mathbb Z \times \mathbb Z \to \mathbb R$ by $f(m,n) = mr + ns$. Now $f$ is an injective homomorphism and the image of $f$ is countable and dense in $\mathbb R$. The second statement follows by Hurwitz's theorem, which states that there are infinitely many pairs $(m,n) \in \mathbb Z \times \mathbb Z$ with
\[
\left|\frac{m}{n} - \frac{s}{r}\right| < \frac{1}{n^2} \ .
\]
Now the pullback norm $|(m,n)| \mathrel{\mathop:}= |f(m,n)|$ on $\mathbb Z \times \mathbb Z$ has \eqref{czt} but is not complete.
\end{example}

\subsection{Conclusion of the classification}
Our classification of complete groups with \eqref{czt} is concluded by the following:
\begin{proof}[Proof of Proposition~\ref{class_czt}]
Assume by contradiction that $G$ contains a torsion element $g_T \in G \setminus \{0\}$ and a non-torsion element $g \in G \setminus \{0\}$. By using the classification from Proposition~\ref{classtorsion}, and up to taking another element if necessary in case $\langle g_T\rangle \simeq \mathbb Z_4$, we can assume that $2g_T = 0$. Together with Proposition~\ref{classtorsionfree} we obtain that $\langle g_T,g \rangle$ is isomorphic to $\mathbb Z_2 \times \mathbb Z$ and as a subgroup of $G$ we obtain a norm $|\cdot|$ on $\mathbb Z_2 \times \mathbb Z$ that has collinear zero-mean triples \eqref{czt}.

\medskip

\noindent Like in the proof of Proposition~\ref{classtorsion}, for a zero-mean triple $a,a, -2a$ with $a, 2a\neq 0$ the only possible collinearity equation is $|a|+|a|=|2a|$. By considering the zero-mean triples $(0,2^k)$, $(0,2^k)$, $(0,-2^{k+1})$ and $(1,2^k)$, $(1,2^k)$, $(0,-2^{k+1})$ we find by induction on $k$ that $2^{k}|(0,1)|=|(0,2^{k})|=2^k|(1,1)|$ for all $k\ge 1$. So $(1,2^k)$, $(0,-2^k)$, $(1,0)$ form a zero-mean triple with norms of the form $2^k\alpha$, $2^k\alpha$, $\beta$, where $\alpha\mathrel{\mathop:}=|(1,1)|=|(0,1)|, \beta\mathrel{\mathop:}=|(1,0)|$, which can't be collinear for $2^{k+1} > \frac{\beta}{\alpha}$. Thus $\mathbb Z_2\times \mathbb Z$ has no norm for which \eqref{czt} holds. 

\medskip

\noindent This implies that $G$ is either a torsion group or torsion-free. Both cases have already been classified in Proposition~\ref{classtorsion} and Proposition~\ref{classtorsionfree}.
\end{proof}
\noindent With this done we are ready to prove Theorem~\ref{class_cotp}, the main theorem of this section.
\begin{proof}[Proof of Theorem~\ref{class_cotp}]
Due to Lemma~\ref{collinearimplication} and Proposition~\ref{class_czt}, the only possible groups that have acyclic nonbranching optimal transport plans are $\mathbb R$, $\mathbb Z$, $\mathbb Z_2$, $\mathbb Z_2\times \mathbb Z_2$, $\mathbb Z_4$. It is shown in Lemma~\ref{check_groups} that except for $\mathbb Z_4$ all these groups have acyclic nonbranching optimal transport plans. So it remains to exclude $\mathbb Z_4$. But $\mathbb Z_4$ doesn't have nonbranching optimal transport plans by the classification for finitely generated groups in Theorem~\ref{class_otp}.
\end{proof}

\section{Nonbranching transport and calibrations, and proof of Proposition~\ref{calibprop}}
\noindent By combining the main result of \cite{pz}, which gives calibrations for $1$-chains with coefficients in $G=\mathbb Z_2$, with the classically known calibration/duality available for $G=\mathbb R$ and $G=\mathbb Z$, we give now a general version of calibrations/Kantorovich duality for groups as in the theorem above.

\medskip
\noindent The basic example from \cite[Remark~2.6]{pz} indicates that calibrations, i.e.\ the possibility to re-express the filling problem as a global dual problem defined in terms of maximization among some class of Lipschitz functions, would be prohibited in the cases where the minimum fillings are branched. Note that for the classical branched transport problem, i.e.\ for the case of the group $(\mathbb R,|\cdot|^\alpha)$, $\alpha\in\ ]0,1[$, the so-called landscape functions, which may be seen as a partial analogue of a calibrations, were introduced in \cite{santambrogio} and \cite{xia2}. However, two properties which would be desirable in order to have a global dual problem to the filling problem are missing in that case: First, the fact that only H\"older (and not Lipschitz) regularity holds for the landscape functions indicates that they do not correspond to a true dual variational problem. Second, the fact that a given landscape function is defined in terms of the branched transport minimizers and not in terms of the sources and weights only, indicates that a given landscape function is only a locally dual object, i.e.\ it will not simultaneously calibrate multiple branched transport minimizers.

\medskip 

\noindent Note that for any normed Abelian group $G$, any geodesic tree $T$ and any $S \in \mathcal R_1(T;G)$, it holds that $\partial S = 0$ implies $S = 0$. This follows directly from the homotopy formula for chains, \cite[\S 2.6]{depauw}, and from the fact that $\mathcal H^2(f(B)) = 0$ for all Lipschitz maps $f : B \to T$ defined on a Borel set $B \subset \mathbb R^2$, which follows for example from \cite[Lemma~3.6]{W}. As a consequence, for any $R \in \mathcal R_0(T;G)$, the filling length $\rm{Fill}_{G,T}(R)$ is achieved by \textit{any} rectifiable filling of $R$, and in this sense the minimum filling problem in trees is trivialized. We next give a proof of Proposition~\ref{calibprop} stated in the introduction. It essentially tells that a filling problem with coefficients in $\mathbb R^k\times \mathbb Z^l\times \mathbb Z_2^k$ can be calibrated by a multivalued map into a tree.

\begin{proof}[Proof of Proposition~\ref{calibprop}]
First note that since we have endowed $G$ with the $\ell_1$-norm, we obtain for any choice of $f_i$ and $T$ as in \eqref{filling_eq} and $S \in \mathcal R_1(X;G)$ with $\partial S = T$ that
\begin{align*}
\rm{Fill}_{G,T}\biggl(\sum_{j=1}^{k+l+m} f_{j\#} (\pi_j R) \biggr) & \leq \mathbf M\biggl(\sum_{j=1}^{k+l+m} f_{j\#} (\pi_j S) \biggr) \leq \sum_{j=1}^{k+l+m} \mathbf M(f_{j\#} (\pi_j S)) \\
 & \leq \sum_{j=1}^{k+l+m} \mathbf M(\pi_j S) = \mathbf M(S) \ .
\end{align*}
Taking the infimum over all such $S$, this shows one inequality in \eqref{filling_eq}.

\medskip

\noindent To obtain the opposite inequality note that for each $\pi_i R$ we can find a finite geodesic tree $T_i$ and a $1$-Lipschitz map $f_i : X \to T_i$ such that $\rm{Fill}_{G,X}(\pi_i R)=\rm{Fill}_{G,T_i}(f_{i\#}(\pi_i R))$. For $G_i=\mathbb R$ or $G_i=\mathbb Z$ we can actually take $T_i = \mathbb R$, or a closed interval, by Kantorovich-duality (of which a version adapted to the present setting is stated e.g.\ in \cite[Theorem~1.3]{pz}). For $G_i=\mathbb Z_2$ this follows from the main result \cite[Theorem~1.4]{pz}, respectively, its formulation for chains in \cite[Proposition~1.6]{pz}. Now one may obtain a finite geodesic tree $T$ by gluing together all the $T_i$. We can actually manage to glue them to a star-shaped tree such that two different $T_i$ and $T_j$ have enough distance inside $T$ that
\[
\rm{Fill}_{G,T}\biggl(\sum_{i=1}^{k+l+m} f_{i\#} (\pi_i R) \biggr) = \sum_{i=1}^{k+l+m} \rm{Fill}_{G,T}\biggl(f_{i\#} (\pi_i R) \biggr) \ .
\]
Hence again using the definition of the $\ell_1$-norm on $G$,
\begin{align*}
\rm{Fill}_{G,T}\biggl(\sum_{i=1}^{k+l+m} f_{i\#} (\pi_i R) \biggr) & = \sum_{i=1}^{k+l+m} \rm{Fill}_{G,X}(\pi_i R) = \rm{Fill}_{G,X}\biggl(\sum_{i=1}^{k+l+m} \pi_i R\biggr) \\
 & = \rm{Fill}_{G,X}(R) \ .
\end{align*}
This concludes the proof of the proposition.
\end{proof}
\noindent There is a partial converse to this statement that generalizes \cite[Remark~2.6]{pz} and essentially tells that only groups with nonbranching optimal transport plans can be calibrated with maps into trees. In the following Lemma, a tree is a tree $T$ together with a Lipschitz path connected metric $d$ (so we don't assume $T$ to be geodesic). As stated before Proposition~\ref{calibprop} we get that $\partial S = 0$ implies $S = 0$ in case $S \in \mathcal R_1(T;G)$.
\begin{lemma}
\label{converse_lem}
Let $G$ be a normed Abelian group with optimal transport plans. Assume that for any $R = \sum_{i=1}^n g_i \curr{x_i} \in \mathcal R_0(X;G)$, where $X$ is a geodesic metric space and $\sum_{i=1}^n g_i = 0$, there exists a tree $(T,d)$ and a $1$-Lipschitz map $f : X \to T$ such that $\rm{Fill}_{G,T}(f_\# R) = \rm{Fill}_{G,X}(R)$. Then $G$ has nonbranching optimal transport plans.
\end{lemma}

\begin{proof}
Consider $g_1,\dots,g_n \in G \setminus \{0\}$ as in the statement and let $(X,d_X)$ be the geodesic metric space obtained by gluing intervals of length $2$ between any two different points of the set $X \mathrel{\mathop:}= \{x_1,\dots,x_n\}$. Since $G$ has optimal transport plans, there exists $S \in \mathcal R_1(X;G)$ with $\partial S = R$ and $\mathbf M(S) = \rm{Fill}_{G_X}(R)$. With the discussion following Definition~\ref{fillingsongeodesics}, $S = \sum_{i < j} g_{ij} \curr{x_j,x_i}$, where $g_{ij} = -g_{ji}$, $g_{ii}=0$ and $g_i = \sum_j g_{ij}$. Then
\begin{align*}
\rm{Fill}_{G,X}(R) & = \mathbf M(S) = \sum_{1 \leq i < j \leq n} |g_{ij}|\ \rm{length}([x_j,x_i]) \\
 & \geq \sum_{1 \leq i < j \leq n} |g_{ij}|\ \rm{length}([f(x_j),f(x_i)]) \\
 & \geq \mathbf M(f_\# S) \\
 & = \rm{Fill}_{G,T}(f_\# R) \ .
\end{align*}
By assumption, equalities hold and hence $f$ maps each segment $[x_j,x_i]$ injectively and length preserving into $T$. Because $T$ is uniquely arcwise connected, there is a unique point $x \in T$ such that $x \in [f(x_j),f(x_i)]$ for different $i,j$ and $\rm{length}([f(x_j),x]) = 1$. Thus
\begin{align*}
\sum_{i =1}^n \sum_{j =1}^n |g_{ij}| & = \mathbf M(S) = \mathbf M(f_\# S) = \sum_{i=1}^n |g_{i}|\ \rm{length}([x,f(x_i)]) = \sum_{i=1}^n |g_{i}| \ .
\end{align*}
Since $|g_{i}| \leq \sum_{j =1}^n |g_{ij}|$ for all $i$ by the triangle inequality, this implies that $|g_{i}| = \sum_{j =1}^n |g_{ij}|$ for all $i$. Hence $G$ has nonbranching optimal transport plans.
\end{proof}

\end{document}